\documentclass[10pt,a4paper]{amsart}

\usepackage{lmodern}
\usepackage[utf8]{inputenc}
\usepackage[T1]{fontenc}

\usepackage{fancyref}
\usepackage[colorlinks=true]{hyperref}
\usepackage{xcolor}

\usepackage{amsmath}
\usepackage{amssymb}
\usepackage{amsthm}
\usepackage{tikz-cd}

\newtheorem{theo}{Theorem}
\newtheorem{prop}[theo]{Proposition}

\newtheorem{coro}[theo]{Corollary}
\theoremstyle{definition}

\newtheorem*{rema}{Remark}

\def\^#1{^{[#1]}}
\def\<{\langle}
\def\>{\rangle}
\def\ov#1{\overline{#1}}
\def\CC{\mathbf{C}}
\def\d{\,d}
\def\herm{\mathbf{H}}

\DeclareMathOperator{\Sym}{Sym}
\DeclareMathOperator{\id}{Id}
\DeclareMathOperator{\tr}{tr}
\DeclareMathOperator{\Vol}{Vol}
\DeclareMathOperator{\End}{End}

\allowdisplaybreaks

\author{Gunnar \TH\'or Magn\'usson}
\address{Hafnarfjörður, Iceland}
\email{gunnar@magnusson.io}
\date{\today}
\title{Integrating holomorphic sectional curvatures}

\begin{document}

\begin{abstract}
We calculate the $L^2$-norm of the holomorphic sectional curvature of a
K\"ahler metric by representation-theoretic means. This yields a new proof that
the holomorphic sectional curvature determines the whole curvature tensor.
We then investigate what the holomorphic sectional curvature of a Hermitian
metric determines and calculate the $L^2$-norm of the holomorphic bisectional
curvature.
\end{abstract}

\maketitle

\section*{Introduction}

Let $X$ be a complex manifold of dimension $n$ and let $h$ be a K\"ahler metric on $X$.
We denote by $R$ the curvature tensor of the Chern connection of $h$.
One of the main ways to simplify this complicated tensor is to consider the holomorphic sectional curvature of $h$, that is
$$
H(\xi) = R(\xi, \ov{\xi}, \xi, \ov{\xi}) / |\xi|^4
$$
for nonzero tangent fields $\xi$.
If $s$ is the scalar curvature of $h$, it is well-known that
$$
\frac{1}{\Vol S(T_{X,x})} \int_{S(T_{X,x})}
\!\!\!
H(\xi) \d\sigma
= \frac{s}{\binom{n+1}{2}},
$$
where $S(T_{X,x})$ is the unit ball in $T_{X,x}$
\cite{berger1965varietes}.
It is also well-known that the holomorphic sectional curvature determines the whole curvature tensor, in the sense that if $H = 0$ then $R = 0$
\cite{zheng2000complex}.

The existing proofs of these facts are by brute force calculations.
The integral over the sphere is broken up into polynomial components that are evaluated separately and summed together.
This works out in the end because of the seeming coincidence that the integrals of $|z_j|^4$ and $|z_j|^2 |z_k|^2$ over the unit sphere agree up to a factor of 2.
The proof of the determination of the curvature tensor is by clever algebraic manipulation that leaves at least myself no wiser as to why this is true at all.

In this article we propose a new route to these facts.
Our starting point is a representation-theoretic identity that is known to
quantum information theorists, namely that
$$
\frac{1}{\Vol S(V)} \int_{S(V)} (v \otimes v^*)^{\otimes d} \d\sigma
= \frac{1}{\binom{n+d-1}{d}} \Pi_{d},
$$
where $V$ is an $n$-dimensional complex vector space, $\Pi_{d} : V^{\otimes
d} \to V^{\otimes d}$ is the projection onto the subspace of symmetric
$d$-tensors and $v^* := u \mapsto h(u, \bar v)$.
Berger's scalar curvature identity is an immediate consequence of this
identity, and it and some routine linear algebra let us also evaluate the $L^2$
norm of the holomorphic sectional curvature and see that
$$
\frac{1}{\Vol S(T_{X,x})} \int_{S(T_{X,x})}
\!\!\!
H(\xi)^2 \d\sigma
= \frac{|R|^2 + 4|r|^2 + s^2}{\binom{n+1}{2}\binom{n+3}{2}},
$$
where $r$ is the Ricci-tensor of $h$, which neatly explains why the holomorphic
sectional curvature determines the whole tensor.

Slightly more bookkeeping lets us also consider the same integrals for the
curvature tensors of Hermitian metrics. There it is no longer true that the
holomorphic sectional curvature determines the whole tensor, but we can say
what it does determine, which turns out to be essentially the ``symmetric''
part of the tensor when viewed as a Hermitian form on $V^{\otimes 2} =
\bigwedge^{2} V \oplus \Sym^2 V$.
As an application we then calculate the $L^2$-norm of the holomorphic
bisectional curvature of a K\"ahler metric.

\subsection*{Acknowledgements}

Many thanks to Kyle Broder for his excellent comments on an earlier version of
this note and for several references to the literature.

\section{Algebraic curvature tensors}

Let $V$ be a complex vector space of dimension $n$, which we think of as the
tangent space of a complex manifold at a given point.
The curvature tensor $R$ of a Hermitian metric on the manifold identifies with
a Hermitian form $q$ on $V \otimes V$, defined by
$$
R(x, \ov y, z, \ov w)
= q(x \otimes z, \ov{y \otimes w}).
$$
If the metric is K\"ahler we get an additional symmetry
$R(x, \ov y, z, \ov w) = R(z, \ov y, x, \ov w)$
(and the ones induced by conjugating).
A nice alternate reference for what we discuss here is
\cite{algebraic-kahler-curvature}.

We write $\herm(V)$ for the real vector space of Hermitian forms on $V$.
The curvature tensor of a Hermitian metric is then just a member of $\herm(V
\otimes V)$. We call such an element an \emph{algebraic Hermitian curvature
tensor}, and one that satisfies the additional symmetry of a K\"ahler curvature
tensor an \emph{algebraic K\"ahler curvature tensor}.

The decomposition $V \otimes V = \bigwedge^2 V \oplus \Sym^2 V$ is standard.
It implies that a Hermitian form on $V \otimes V$ decomposes into components
$$
q = \begin{pmatrix}
q_{\wedge^2 V} & q_{(\Sym^2V, \wedge^2 V)}
\\
q_{(\wedge^2 V, \Sym^2V)} & q_{\Sym^2 V},
\end{pmatrix}
$$
where $q_{(\wedge^2 V, \Sym^2V)}^\dagger = q_{(\Sym^2V, \wedge^2 V)}$.
Denote the symmetrization map by
$$
\Pi_2 : V \otimes V \to \Sym^2 V,
\quad
x \otimes y \mapsto \tfrac 12 (x \otimes y + y \otimes x).
$$
It is a surjective linear morphism that realizes the space of symmetric tensors
as a subspace of $V \otimes V$.
The usual definition of that space is as a quotient of $V \otimes V$ by the
ideal generated by $x \otimes y - y \otimes x$.
As the field $\CC$ has characteristic zero these spaces are isomorphic,
so the difference between them isn't very important to us.

\begin{prop}
A tensor $R \in \herm(V \otimes V)$ is an algebraic K\"ahler curvature tensor
if and only if there exists an element $\hat R \in \herm(\Sym^2 V)$ such that $R =
\Pi_2^* \hat R$.
\end{prop}

\begin{proof}
Suppose $R \in \herm(V \otimes V)$ is an algebraic K\"ahler curvature tensor.
We define
$$
\hat R(x \odot z, \ov{y \odot w})
= R(x, \ov y, z, \ov w),
$$
which is well-defined because $R$ is K\"ahler, and we have $R = \Pi_2^* \hat R$.

Conversely, let $\hat R \in \herm(\Sym^2 V)$, and define $R = \Pi_2^* \hat R$.
Then
$$
R(x, \ov y, z, \ov w)
= \hat R(x \odot z, \ov{y \odot w})
= \hat R(z \odot x, \ov{w \odot y})
= R(z, \ov y, x, \ov w)
$$
is an algebraic K\"ahler curvature tensor.
\end{proof}

As an aside, this explains why we only ever talk about Griffits
positivity of K\"ahler metrics.
A Hermitian metric is Nakano positive if its curvature tensor is
positive-definite as a Hermitian form on $V \otimes V$.
A form that's a pullback by a morphism with nontrivial
kernel is never positive-definite.
Therefore a K\"ahler metric can only be Nakano positive when $\Pi_2$ is
injective, which happens when $\ker \Pi_2 = \bigwedge^2 V = 0$,
that is, when $\dim V = 1$.

There should then be a notion of positivity for K\"ahler curvature tensors that
interpolates between Griffiths positivity and Nakano positivity and is perhaps
more geometrically motivated than Griffiths positivity, where we would say that
such a tensor is positive if its Hermitian form on $\Sym^2 V$ is
positive-definite.
However such a metric would also be Griffiths-positive, so Siu and Yau's
theorem would imply that the underlying space is a projective space. This
should offer something different from $m$-positivity (see~\cite[Chapter~7,
Definition~6.5]{agbook}) that might be interesting in the negatively curved
case.

\section{Projection formula}

Our starting point involves vector-valued integration, so let's recall
some basic facts.
Let $V$ and $W$ be finite-dimensional complex vector spaces equipped with their
Lebesgue measures.
If $f : V \to W$ is a continuous function and $X \subset V$ a measurable subset
we define $\int_X f(v) \d\mu(v) \in W$ to be the vector we get after picking
bases and integrating coordinate by coordinate.
If $L : W \to Z$ is a linear map then $L \int_X f(v) \d\mu(v) = \int_X Lf(v)
d\mu(v)$, because that holds at every step of the definition of the integral of
a measurable function.
This implies the integral is independent of the choice of basis.
It also implies that if $f : \End V \to \End V$ is continuous then the trace
commutes with the integral.

Let's suppose $V$ is equipped with a Hermitian inner product $h$
and let's write $d\mu$ for the induced volume form on the unit sphere $S(V)$ in
$V$. For $v \in V$ we define $v \otimes v^*$ to be the linear map
$h^*(\ov v) v = x \mapsto h(x, \ov v) v$.
Note that if $f \in \End V$ then $f \circ v \otimes v^* = h^*(\ov v) f(v)$
and thus $\tr(f \circ v \otimes v^*) = h(fv, \ov v)$.
The following is known to quantum information theorists; see
\cite{harrow2013church}
and
\cite[Chapter~7]{watrous2018theory}.

\begin{prop}
\label{prop:projection}
Denote by $\Pi_d : V^{\otimes d} \to V^{\otimes d}$ the projection onto $\Sym^d V$.
Then
$$
\frac{1}{\Vol S(V)}
\int_{S(V)} (v \otimes v^*)^{d} \d\mu(v)
= \frac{1}{\binom{n+d-1}{d}} \Pi_d.
$$
\end{prop}

\begin{proof}
Let's denote the map defined by the integral by $L$. If $g \in \End V$ is
unitary with respect to $h$ we have $h(g x, \ov v) = h(x, \ov{g^{-1} v})$ and
$|\!\det g| = 1$ so the change of basis formula implies that $L g^d = g^d L$.
The map $L$ is then an interleaving operator of the representation $V^{\otimes d}$
of the unitary group.

The integral is thus a sum of multiples of the projections onto the irreducible
factors of $V^{\otimes d}$.
Note however that the integral takes values in $\Sym^d V$ because the integrand
is invariant under the symmetric group $S_d$, so $L = \lambda \Pi_d$ for some
scalar $\lambda$. For a unit vector $v$ we have $\tr v \otimes v^* = |v|^2 =
1$.
Taking the trace we then find that $1 = \lambda \binom{n+d-1}{d}$,
which implies the result.
\end{proof}

\section{K\"ahler metrics}

We work locally on the manifold $X$ and will write $V = T_{X,x}$.
We suppose here that $h$ is a K\"ahler metric, which just means that we have a
Hermitian inner product on $V$, and that the associated algebraic curvature
tensor $R$ arises as the pullback of a tensor in $\herm(\Sym^2 V)$.
(In particular our results also apply to K\"ahler-like metrics; see
\cite{yang2016curvature}.)

\begin{coro}
$$
\frac{1}{\Vol S(V)} \int_{S(V)}
\!\!\!
H(v)  \d\mu(v)
= \frac{s}{\binom{n+1}{2}}.
$$
\end{coro}

\begin{proof}
If $f \in \End V^{\otimes 2}$ then taking the trace of the
equation in Proposition~\ref{prop:projection} for $d = 2$ gives
$$
\frac{1}{\Vol S(V)}
\int_{S(V)}
\!\!\!
h^{\otimes 2}(f(v \otimes v), \ov{v \otimes v}) \d\mu(v)
= \frac{1}{\binom{n+1}{2}} \tr(f \circ \Pi_2).
$$
If $q \in \herm(\Sym^2)$ is the Hermitian form defined by the curvature tensor $R$,
we apply this to $f = (h^{\otimes 2})^{-1} \Pi_2^* q$.
As $\Pi_2^2 = \Pi_2$ the result follows.
\end{proof}

\begin{coro}
$$
\frac{1}{\Vol S(V)} \int_{S(V)}
\!\!\!
H(v)^2  \d\mu(v)
= \frac{|R|^2 + 4|r|^2 + s^2}{\binom{n+1}{2} \binom{n+3}2}.
$$
\end{coro}

\begin{proof}
We consider $f \otimes f \in \End V^{\otimes 4}$ where
$f$ is as before and get that
$$
\frac{1}{\Vol S(V)}
\int_{S(V)}
\!\!\!
h^{\otimes 2}(f(v \otimes v), \ov{v \otimes v})^2 \d\mu(v)
= \frac{1}{\binom{n+3}{4}}
\tr (f \otimes f \circ \Pi_4).
$$
We know that $\Pi_d = \frac{1}{d!}\sum_{\sigma \in S_d} W_\sigma$, where $S_d$ is the
symmetric group on $d$ letters and $W_\sigma(v_1 \otimes \cdots \otimes v_d) =
v_{\sigma(1)} \otimes \cdots \otimes v_{\sigma(d)}$.
In what follows we're going to recall that $f$ is Hermitian;
use Einstein notation for sums;
use the partial trace
functions $\tr_{jk} : \End V^{\otimes 4} \to \End V^{\otimes 2}$ defined by
taking the trace along indices $j$ and $k$;
and refer to the Frobenius inner product on endomorphisms.
When we write out the elements of $S_4$ as cycles on the letters $(jklm)$
we get the 24 traces
\begin{align*}
\tr(f \otimes f \circ W_{(jklm)})
&= f_{jk,jk} f_{lm,lm}
= \ov{\tr \tr_{13} f} \tr \tr_{13} f
= (\tr f)^2,
\\
\tr(f \otimes f \circ W_{(jkml)})
&= f_{jk,jk} f_{ml,lm}
= \ov{\tr \tr_{13} f} \tr \tr_{14} f
= \tr(\tr_{14} f) \, \tr f,
\\
\tr(f \otimes f \circ W_{(jlkm)})
&= f_{jl,jk} f_{km,lm}
= \ov{(\tr_{13} f)_{kl}} (\tr_{24} f)_{kl}
= \langle \tr_{24} f, \ov{\tr_{13} f} \rangle,
\\
\tr(f \otimes f \circ W_{(jlmk)})
&= f_{jl,jk} f_{mk,lm}
= \ov{(\tr_{13} f)_{kl}} (\tr_{14} f)_{kl}
= \langle \tr_{14} f, \ov{\tr_{13} f} \rangle,
\\
\tr(f \otimes f \circ W_{(jmkl)})
&= f_{jm,jk} f_{kl,lm}
= \ov{(\tr_{13} f)_{km}} (\tr_{23} f)_{km}
= \langle \tr_{23} f, \ov{\tr_{13} f} \rangle,
\\
\tr(f \otimes f \circ W_{(jmlk)})
&= f_{jm,jk} f_{lk,lm}
= \ov{(\tr_{13} f)_{km}} (\tr_{13} f)_{km}
= |\!\tr_{13} f|^2,
\\
\tr(f \otimes f \circ W_{(kjlm)})
&= f_{kj,jk} f_{lm,lm}
= \ov{\tr \tr_{14} f} \tr \tr_{13} f
= \tr(\tr_{14}f) \, \tr f,
\\
\tr(f \otimes f \circ W_{(kjml)})
&= f_{kj,jk} f_{ml,lm}
= \ov{\tr \tr_{14} f} \tr \tr_{14} f
= \tr(\tr_{14} f)^2,
\\
\tr(f \otimes f \circ W_{(kljm)})
&= f_{kl,jk} f_{jm,lm}
= \ov{(\tr_{23} f)_{jl}} (\tr_{24} f)_{jl}
= \langle \tr_{24} f, \ov{\tr_{23} f} \rangle,
\\
\tr(f \otimes f \circ W_{(klmj)})
&= f_{kl,jk} f_{mj,lm}
= \ov{(\tr_{23} f)_{jl}} (\tr_{14} f)_{jl}
= \langle \tr_{14} f, \ov{\tr_{23} f} \rangle,
\\
\tr(f \otimes f \circ W_{(kmjl)})
&= f_{km,jk} f_{jl,lm}
= \ov{(\tr_{23} f)_{jm}} (\tr_{23} f)_{jm}
= |\! \tr_{23} f|^2,
\\
\tr(f \otimes f \circ W_{(kmlj)})
&= f_{km,jk} f_{lj,lm}
= \ov{(\tr_{23} f)_{jm}} (\tr_{13} f)_{jm}
= \langle \tr_{13} f, \ov{\tr_{23} f} \rangle,
\\
\tr(f \otimes f \circ W_{(ljkm)})
&= f_{lj,jk} f_{km,lm}
= \ov{(\tr_{14} f)_{kl}} (\tr_{24} f)_{kl}
= \langle \tr_{24} f, \ov{\tr_{14} f} \rangle,
\\
\tr(f \otimes f \circ W_{(ljmk)})
&= f_{lj,jk} f_{mk,lm}
= \ov{(\tr_{14} f)_{kl}} (\tr_{14} f)_{kl}
= |\! \tr_{14} f|^2,
\\
\tr(f \otimes f \circ W_{(lkjm)})
&= f_{lk,jk} f_{jm,lm}
= \ov{(\tr_{24} f)_{jl}} (\tr_{24} f)_{jl}
= |\! \tr_{24} f|^2,
\\
\tr(f \otimes f \circ W_{(lkmj)})
&= f_{lk,jk} f_{mj,lm}
= \ov{(\tr_{24} f)_{jl}} (\tr_{14} f)_{jl}
= \langle \tr_{14} f, \ov{\tr_{24} f} \rangle,
\\
\tr(f \otimes f \circ W_{(lmjk)})
&= f_{lm,jk} f_{jk,lm}
= \ov{f_{jk,lm}} f_{jk,lm}
= |f|^2,
\\
\tr(f \otimes f \circ W_{(lmkj)})
&= f_{lm,jk} f_{kj,lm}
= \ov{f_{jk,lm}} f_{kj,lm}
= \langle f \circ W_{(12)}, \ov f \rangle ,
\\
\tr(f \otimes f \circ W_{(mjkl)})
&= f_{mj,jk} f_{kl,lm}
= \ov{(\tr_{14} f)_{km}} (\tr_{23} f)_{km}
= \langle \tr_{23} f, \ov{\tr_{14} f} \rangle,
\\
\tr(f \otimes f \circ W_{(mjlk)})
&= f_{mj,jk} f_{lk,lm}
= \ov{(\tr_{14} f)_{km}} (\tr_{13} f)_{km}
= \langle \tr_{13} f, \ov{\tr_{14} f} \rangle,
\\
\tr(f \otimes f \circ W_{(mkjl)})
&= f_{mk,jk} f_{jl,lm}
= \ov{(\tr_{24} f)_{jm}} (\tr_{23} f)_{jm}
= \langle \tr_{23} f, \ov{\tr_{24} f} \rangle,
\\
\tr(f \otimes f \circ W_{(mklj)})
&= f_{mk,jk} f_{lj,lm}
= \ov{(\tr_{24} f)_{jm}} (\tr_{13} f)_{jm}
= \langle \tr_{13} f, \ov{\tr_{24} f} \rangle,
\\
\tr(f \otimes f \circ W_{(mljk)})
&= f_{ml,jk} f_{jk,lm}
= \ov{f_{jk,ml}} f_{jk,lm}
= \langle f, \ov{W_{(12)} \circ f} \rangle,
\\
\tr(f \otimes f \circ W_{(mlkj)})
&= f_{ml,jk} f_{kj,lm}
= \ov{f_{jk,ml}} f_{kj,lm}
= \langle f \circ W_{(12)}, \ov{W_{(12)} \circ f} \rangle.
\end{align*}
As $R$ is a K\"ahler curvature tensor, we have $\tr_{13} f = \tr_{14} f =
\tr_{23} f = r$,
where $r$ is
the Ricci tensor of the metric, and $W_{(12)} \circ f = f = f \circ W_{(12)}$.
Adding everything up we get
$$
\tr(f \circ \Pi_4)
= \frac{1}{24} (4|R|^2 + 16 |r|^2 + 4 s^2)
$$
and the result follows after we notice that
\(
6 \binom{n+3}{4} = \binom{n+1}{2} \binom{n+3}{2}
\).
\end{proof}

\begin{rema}
These formulas let us calculate the variance of the holomorphic sectional
curvature over the unit sphere.
Recall that if $c_j$ are the Chern forms of the metric $h$, $\omega$ is its
K\"ahler form, and $\omega\^k := \omega^k/k!$ we have
\begin{align*}
|r|^2 \omega\^n
&= -c_1^2 \wedge \omega\^{n-2} + s^2 \omega\^n,
\\
|R|^2 \omega\^n
&= (2c_2 - c_1^2) \wedge \omega\^{n-2} + |c_1|^2 \omega\^n
= 2(c_2 - c_1^2) \wedge \omega\^{n-2} + s^2\omega\^n.
\end{align*}
Then we get
$$
0 \leq \operatorname{Var} H \, \omega\^n
= \frac{2(c_2-3c_1^2)\wedge \omega\^{n-2} + 6 s^2 \omega\^n}{\binom{n+1}2\binom{n+3}2}
- \frac{s^2 \omega\^n}{\binom{n+1}{2}^2}
$$
and after some calculations end up with
$$
0 \leq
(c_2-3c_1^2)\wedge \omega\^{n-2}
+ \frac{(5n+6)(n-1)}{2n(n+1)} s^2 \omega\^n.
$$
When $s = 0$ we get $\int_X c_2 \wedge \omega\^{n-2} \geq \int_X 3c_1^2 \wedge
\omega\^{n-2}$ with equality if and only if $h$ is flat.
This makes for a fairly poor detector of zero scalar curvature metrics.
It is known that if a complex surface admits a K\"ahler metric of zero scalar
curvature then the surface is either flat, a K3 surface, or the blowup of a
rational surface, and it is not known whether the last type carries such a metric \cite{pedersen1990kahler}.

On the projective plane blown up in $d$ distinct points we have $c_1^2 = 9 -
d$ and $c_2 = 3$, so $c_2 > 3c_1^2$ if and only if $3 > 27 - 3d$ or $d > 8$.
For $d \leq 8$ there is thus no K\"ahler metric of zero scalar curvature on
the surface.
But we already knew this since those are just del~Pezzo surfaces, whose $c_1$ is
positive, so no K\"ahler metric can have zero scalar curvature on those.

On a Hirzebruch surface blown up in $d$ distinct points we have $c_1^2 = 8 -
d$ and $c_2 = 4$, so $c_2 > 3 c_1^2$ if and only if $4 > 24 - 3d$ or $d > 6$.
For $d \leq 6$ there is thus no K\"ahler metric of zero scalar curvature on
the surface.
But again the resulting surface still has positive $c_1$, so this is not news.
\end{rema}

\section{Hermitian metrics}

We still work locally on the manifold $X$ and still write $V = T_X$.
We now suppose only that $h$ is a Hermitian metric, which again gives a
Hermitian inner product on $V$, but now consider an arbitrary algebraic
Hermitian curvature tensor $R$ on $V \otimes V$.
Recall that such a tensor has four Ricci tensors
\begin{align*}
r_1(u, \bar v) &= \sum_{j=1}^n R(e_j, \bar e_j, u, \bar v),
	       &
r_2(u, \bar v) &= \sum_{j=1}^n R(e_j, u, \bar v, \bar e_j),
\\
r_3(u, \bar v) &= \sum_{j=1}^n R(u, \bar v, e_j, \bar e_j),
	       &
r_4(u, \bar v) &= \sum_{j=1}^n R(u, e_j, \bar e_j, \bar v),
\end{align*}
where $(e_j)$ is an orthonormal basis.
Note that $r_1$ and $r_3$ are Hermitian tensors, and
$\ov{r_2(u, \bar v)} = r_4(v, \bar u)$.
We then get two real-valued scalar curvatures
\begin{align*}
s_1 = \sum_{j,k=1}^n R(e_j, \bar e_j, e_k, \bar e_k),
\quad
s_2 = \sum_{j,k=1}^n R(e_j, e_k, \bar e_k, \bar e_j).
\end{align*}

\begin{coro}
$$
\frac{1}{\Vol S(V)} \int_{S(V)}
\!\!\!
H(v)  \d\mu(v)
= \frac{1}{\binom{n+1}{2}} \frac{s_1 + s_2}{2} .
$$
\end{coro}

\begin{proof}
We consider the Hermitian endomorphism $f = (h^{\otimes 2})^{-1} R$ of $V
\otimes V$.
The projection formula and some calculations give
$$
\frac{1}{\Vol S(V)} \int_{S(V)}
\!\!\!
H(v)  \d\mu(v)
= \frac{1}{\binom{n+1}{2}} \tr(f \circ \Pi_2).
$$
As before we know that $\Pi_2 = \frac12(\id + W_{(12)})$
and we can check that $\tr f = s_1$ and $\tr(f \circ W_{(12)}) = s_2$.
\end{proof}

Recall the splitting $V^{\otimes 2} = \bigwedge^2 V \oplus \Sym^2 V$.
Under it any Hermitian form $R$ on $V^{\otimes 2}$ can be written as
$$
R = \begin{pmatrix}
R_{\wedge^2 V} & R_{(\Sym^2V, \wedge^2 V)}
\\
R_{(\wedge^2 V, \Sym^2V)} & R_{\Sym^2 V}
\end{pmatrix}.
$$
We write $r^\dagger$ for the adjoint of $r$.

\begin{coro}
$$
\frac{1}{\Vol S(V)} \int_{S(V)}
\!\!\!
H(v)^2  \d\mu(v)
= \frac{
\bigl(
4 |R_{\Sym^2 V}|^2
+ |r_1 + r_2 + r_3 + r_4|^2
+ (s_1 + s_2)^2
\bigr)
}{4! \binom{n+3}{4}}.
$$
\end{coro}

\begin{proof}
The proof begins exactly as before and we calculate our way to the 24 traces.
Once there
the sum of the 16 different $\langle \tr_{jk} f, \ov{\tr_{ml} f} \rangle$
factors is $|\!\tr_{13} f + \tr_{14} f + \tr_{23} f + \tr_{24} f|^2$.
Note that as $f$ is Hermitian we have
$$
(\tr_{13} f)^\dagger = \tr_{13} f,
\quad
(\tr_{14} f)^\dagger = \tr_{23} f,
\quad
(\tr_{24} f)^\dagger = \tr_{24} f
$$
and the four Ricci forms of $R$ are $r_1 = \tr_{13} f$, $r_2 = \tr_{14} f$,
$r_3 = \tr_{24} f$ and $r_4 = \tr_{23} f$.
This gives the middle factor above.

The sum of the 4 different total trace factors is likewise
$$
(\tr f + \tr(\tr_{14} f))^2 = (s_1 + s_2)^2
$$
(as the total traces as real because $f$ is Hermitian).
This gives the third factor.

That leaves the sum
$$
|f|^2
+ \< f \circ W_{(12)}, \ov f \>
+ \< f, \ov{W_{(12)} \circ f} \>
+ \< f \circ W_{(12)}, \ov{W_{(12)} \circ f} \>.
$$
Write $\sigma = W_{(12)} \in \End V^{\otimes 2}$.
Clearly $\sigma^2 = \id$ so its eigenvalues are $1$ and $-1$, and in fact
the splitting $V^{\otimes 2} = \Sym^2 V \oplus \bigwedge^2 V$
is according to the eigenspaces of $\sigma$,
which acts as $\id$ on $\Sym^2V$ and $-\id$ on $\bigwedge^2 V$.
Writing $\sigma$ out according to this splitting we have
$$
\sigma = \begin{pmatrix}
\id & 0
\\
0 & -\id
\end{pmatrix}
$$
from which it is clear that $\sigma^\dagger = \sigma$.
Because of this and $\sigma^2 = \id$ we get
$$
\< f \sigma, \ov g \> = \<f, \ov{g\sigma}\>,
\quad
\< \sigma f , \ov g \> = \<f, \ov{\sigma g}\>,
\quad
\< \sigma f, \ov{\sigma g} \> = \<f, \ov g\>
$$
for any endomorphisms $f,g$ of $V^{\otimes 2}$.
Note that $\Pi_2 = \frac12(\id + \sigma)$.
For any $f \in \End V^{\otimes 2}$ its symmetric part $f_{\Sym}$ according to
the splitting $V^{\otimes 2} = \bigwedge^2 V \oplus \Sym^2 V$ is $\Pi_2 \circ f
\circ \Pi_2$ and
$$
\Pi_2 \circ f \circ \Pi_2
= \tfrac14 (\id + \sigma) f (\id + \sigma)
= \tfrac14 (f + f \sigma + \sigma f + \sigma f \sigma).
$$
Then
\begin{align*}
16 |\Pi_2 \circ f \circ \Pi_2|^2
&= |f|^2
+ \< f, f \sigma \>
+ \< f, \sigma f \>
+ \< f, \sigma f \sigma \>
\\
&\qquad
+ \< f \sigma, f \>
+ \< f \sigma, f \sigma \>
+ \< f \sigma, \sigma f \>
+ \< f \sigma, \sigma f \sigma \>
\\
&\qquad
+ \< \sigma f, f \>
+ \< \sigma f, f \sigma \>
+ \< \sigma f, \sigma f \>
+ \< \sigma f, \sigma f \sigma \>
\\
&\qquad
+ \< \sigma f \sigma, f \>
+ \< \sigma f \sigma, f \sigma \>
+ \< \sigma f \sigma, \sigma f \>
+ \< \sigma f \sigma, \sigma f \sigma \>
\\
&= |f|^2
+ \< f \sigma, f \>
+ \< \sigma f, f \>
+ \< \sigma f \sigma, f \>
\\
&\qquad
+ \< f \sigma, f \>
+ |f|^2
+ \< \sigma f \sigma, f \>
+ \< \sigma f , f \>
\\
&\qquad
+ \< \sigma f, f \>
+ \< \sigma f \sigma, f \>
+ |f|^2
+ \< f\sigma , f \>
\\
&\qquad
+ \< \sigma f \sigma, f \>
+ \< \sigma f, f \>
+ \< f \sigma, f \>
+ |f|^2
\\
&= 4(|f|^2
+ \< f \sigma, f \>
+ \< \sigma f, f \>
+ \< \sigma f \sigma, f \>).
\end{align*}
Then finally
$$
\frac{1}{\Vol S(V)}
\int_{S(V)}
\!\!\!
H(v)^2 \, d\mu(v)
= \frac{
\bigl(
4 |R_{\Sym^2 V}|^2
+ |r_1 + r_2 + r_3 + r_4|^2
+ (s_1 + s_2)^2
\bigr)
}{4! \binom{n+3}{4}}
$$
as announced.
\end{proof}

So what can we say about the curvature tensor $R$ of a Hermitian metric that has
zero holomorphic sectional curvature?
First off, the symmetric part $R_{\Sym^2 V} = 0$, so according to the splitting
above we can write
$$
R = \begin{pmatrix}
R_{\wedge^2 V} & R_{(\Sym^2V, \wedge^2 V)}
\\
R_{(\wedge^2 V, \Sym^2V)} & 0
\end{pmatrix}.
$$
We also see that $s_1 = -s_2$, and that $r_1 + r_2 + r_3 + r_4 = 0$.
It's interesting that we could have scalar curvatures of differing sign
associated to the same connection (Balas~\cite{balas1985compact} already noted
that $s_1 + s_2 = 0$ in this case), but apart from that it's not clear that
there is further water to be squeezed from this stone.

\section{Holomorphic bisectional curvature}

The holomorphic bisectional curvature of a K\"ahler metric $h$ with curvature
tensor $R$ is
$$
B(\xi,\eta) = \frac{R(\xi, \ov\xi, \eta, \ov\eta)}{|\xi|^2 |\eta|^2}.
$$
Berger noted that it dominates the Ricci curvature by calculating
$$
\frac{1}{\Vol S(T_X)} \int_{S(T_X)} \!\!\! B(\xi, \eta) \, d\sigma(\xi)
= \frac{1}{n} r(\eta, \ov\eta).
$$
Kyle Broder asked if we could calculate the $L^2$-norm of the holomorphic
bisectional curvature using these methods, and we can.

\begin{coro}
$$
\frac{1}{\Vol S(V)^2}
\int_{S(V)^2} \!\!\! B(u, v)^2 \, d\sigma(u) d\sigma(v)
= \frac{|R|^2 + (n+2)|r|^2}{4\binom{n+1}{2}^2}.
$$
\end{coro}

\begin{proof}
This amounts to calculating
$$
\frac{1}{\Vol S(V)^2}
\int_{S(V)^2} \!\!\! R(u, \bar u, v, \bar v)^2 \, d\sigma(u) d\sigma(v)
$$
for an algebraic K\"ahler curvature tensor $R$.
Fix $v$ and let
$$
T_v(a, \bar b, c, \bar d) = R(a, \bar b, v, \bar v) R(c, \bar d, v, \bar v).
$$
Then $T_v$ is an algebraic curvature tensor, that is, a Hermitian form on
$\bigwedge^2 V$.
It does not have the symmetries of a K\"ahler curvature tensor, but we have
$$
\displaylines{
\frac{1}{\Vol S(V)^2}
\int_{S(V)^2} R(u, \bar u, v, \bar v)^2 \, d\sigma(u) d\sigma(v)
\hfill\cr\hfill{}
= \frac{1}{\Vol S(V)} \int_{S(V)}
\biggl(
\frac{1}{\Vol S(V)} \int_{S(V)}
T_v(u, \bar u, u, \bar u) d\sigma(u)
\biggr)
d\sigma(v).
}
$$
As we know,
$$
\frac{1}{\Vol S(V)} \int_{S(V)} T_v(u, \bar u, u, \bar u) d\sigma(u)
= \frac{1}{\binom{n+1}{2}} \frac{s_1(T_v) + s_2(T_v)}{2}.
$$
In an orthonormal basis $(e_j)$ we have
\begin{align*}
s_1(T_v) &= \sum_{j,k=1}^n T_v(e_j, \bar e_j, e_k, \bar e_k)
= \sum_{j,k=1}^n R(e_j, \bar e_j, v, \bar v) R(e_k, \bar e_k, v, \bar v)
= r(v, \bar v)^2,
\\
s_2(T_v) &= \sum_{j,k=1}^n T_v(e_j, \bar e_k, e_k, \bar e_j)
= \sum_{j,k=1}^n |R(e_j, \bar e_k, v, \bar v)|^2.
\end{align*}
Now
$$
\frac{1}{\Vol S(V)} \int_{S(V)} r(v,\bar v)^2 d\sigma(v)
= \frac{|r|^2}{n}
$$
by standard linear algebra.
Playing the same game again with
$$
U(a, \bar b, c, \bar d)
= \sum_{j,k=1}^n R(e_j, \bar e_k, a, \bar b) R(e_k, \bar e_j, c, \bar d)
$$
we get
$$
\frac{1}{\Vol S(V)} \int_{S(V)} s_2(T_v) d\sigma(v)
= \frac{1}{\binom{n+1}{2}} \frac{s_1(U) + s_2(U)}{2}
$$
and
\begin{align*}
s_1(U)
&= \sum_{j,k,l,m=1}^n
R(e_j, \bar e_k, e_l, \bar e_l) R(e_k, \bar e_j, e_m, \bar e_m)
= \sum_{j,k=1}^n
|r(e_j, \bar e_k)|^2
= |r|^2
\\
s_2(U) &= \sum_{j,k,l,m=1}^n
R(e_j, \bar e_k, e_l, \bar e_m) R(e_k, \bar e_j, e_m, \bar e_l)
= |R|^2.
\end{align*}
All together we get
\begin{align*}
\frac{1}{\Vol S(V)^2}
\int_{S(V)^2} B(u,v)^2 \, d\sigma(u) d\sigma(v)
= \frac{1}{2\binom{n+1}{2}}
\biggl(
\frac{|r|^2}{n}
+ \frac{1}{2\binom{n+1}{2}}
\biggl(
|r|^2
+ |R|^2
\biggr)
\end{align*}
which simplifies to what we claimed.
\end{proof}

For Hermitian metrics there is a variety of possible holomorphic bisectional
curvatures to consider; we get four from the various permutations of the
positions of $\xi$ and $\eta$ (these are not real-valued, which seems odd), and
then convex combinations of those (the ones that correspond to averages can be
real-valued).
Given any of these we can repeat the above to calculate its $L^2$-norm and
should get a similar linear combination of the norms of the various scalar and
Ricci curvatures.
We leave this to the interested reader.

\bibliographystyle{plainurl}
\bibliography{main}

\begin{thebibliography}{1}

\bibitem{balas1985compact}
Andrew Balas.
\newblock Compact {H}ermitian manifolds of constant holomorphic sectional
  curvature.
\newblock {\em Mathematische Zeitschrift}, 189:193--210, 1985.
\newblock URL: \url{https://doi.org/10.1007/BF01175044}.

\bibitem{berger1965varietes}
Marcel Berger.
\newblock Sur les vari{\'e}t{\'e}s d’einstein compactes.
\newblock {\em Comptes Rendus de la IIIe R{\'e}union du Groupement des
  Math{\'e}maticiens d’Expression Latine (Namur, 1965)}, pages 35--55, 1965.

\bibitem{agbook}
Jean-Pierre Demailly.
\newblock Complex analytic and differential geometry, 2021.
\newblock URL:
  \url{https://www-fourier.ujf-grenoble.fr/~demailly/manuscripts/agbook.pdf}.

\bibitem{harrow2013church}
Aram~W Harrow.
\newblock The church of the symmetric subspace, 2013.
\newblock URL: \url{https://arxiv.org/abs/1308.6595}.

\bibitem{pedersen1990kahler}
Henrik Pedersen and Yat~Sun Poon.
\newblock K\"ahler surfaces with zero scalar curvature.
\newblock {\em Classical and Quantum Gravity}, 7(10):1707, 1990.

\bibitem{algebraic-kahler-curvature}
Malladi Sitaramayya.
\newblock Curvature tensors in {K}aehler manifolds.
\newblock {\em Trans. Amer. Math. Soc.}, 183:341--353, 1973.
\newblock URL: \url{http://www.jstor.org/stable/1996473}.

\bibitem{watrous2018theory}
John Watrous.
\newblock {\em The theory of quantum information}.
\newblock Cambridge university press, 2018.
\newblock URL: \url{https://cs.uwaterloo.ca/~watrous/TQI/}.

\bibitem{yang2016curvature}
Bo~Yang and Fangyang Zheng.
\newblock On curvature tensors of hermitian manifolds.
\newblock {\em arXiv}, 2016.
\newblock URL: \url{https://arxiv.org/abs/1602.01189}.

\bibitem{zheng2000complex}
Fangyang Zheng.
\newblock {\em Complex differential geometry}.
\newblock Number~18. Amer. Math. Soc., 2000.

\end{thebibliography}

\end{document}